\documentclass[a4paper,12pt]{amsart}

\usepackage{latexsym}
\usepackage{amssymb}
\usepackage{graphicx}
\usepackage{amsthm}
\usepackage{epsfig}
\RequirePackage{color}

\usepackage{amscd,enumerate}
\usepackage{amsfonts}

\input xy
\xyoption{all}

\usepackage{multicol}
\usepackage{hyperref}
\hypersetup{ colorlinks,
linkcolor=blue,
filecolor=green,
urlcolor=blue,
citecolor=blue }

\newtheorem{lemma}{Lemma}[section]
\newtheorem{corollary}[lemma]{Corollary}
\newtheorem{theorem}[lemma]{Theorem}
\newtheorem{proposition}[lemma]{Proposition}
\newtheorem{remark}[lemma]{Remark}
\newtheorem{definition}[lemma]{Definition}
\newtheorem{definitions}[lemma]{Definitions}
\newtheorem{example}[lemma]{Example}
\newtheorem{examples}[lemma]{Examples}

\newtheorem{notation}[lemma]{Notation}

\newcommand{\C}{{\mathcal{C}}}

\definecolor{turquoise2}{rgb}{0,0.898039,0.933333}
\definecolor{magenta}{rgb}{1,0,1}

\begin{document}


\title[Extreme cycles. The center of a Leavitt path algebra]{Extreme cycles. The center of a Leavitt path algebra}

\author[M. G. Corrales]{Mar\'{\i}a G. Corrales Garc\'{\i}a}
\address{M. G. Corrales Garc\'{\i}a:  Centro Regional Universitario de Cocl\'e: \lq\lq Dr. 
Bernardo Lombardo\rq\rq, Universidad de  Panam\'a.  Apartado Postal 0229. Penonom\'e, 
Provincia de Cocl\'e. Panam\'a.}
\email{mcorrales@ancon.up.ac.pa}
\author[D. Mart\'{\i}n]{Dolores Mart\'{\i}n Barquero}
\address{D. Mart\'{\i}n Barquero: Departamento de Matem\'atica Aplicada, Escuela T\'ecnica Superior de Ingenieros Industriales, Universidad de M\'alaga. 29071 M\'alaga. Spain.}
\email{dmartin@uma.es}

\author[C. Mart\'{\i}n]{C\'andido Mart\'{\i}n Gonz\'alez}
\address{C. Mart\'{\i}n Gonz\'alez:  Departamento de \'Algebra Geometr\'{\i}a y Topolog\'{\i}a, Fa\-cultad de Ciencias, Universidad de M\'alaga, Campus de Teatinos s/n. 29071 M\'alaga. Spain.}
\email{candido@apncs.cie.uma.es}

\author[M. Siles ]{Mercedes Siles Molina}
\address{M. Siles Molina: Departamento de \'Algebra Geometr\'{\i}a y Topolog\'{\i}a, Fa\-cultad de Ciencias, Universidad de M\'alaga, Campus de Teatinos s/n. 29071 M\'alaga.   Spain.}
\email{msilesm@uma.es}

\author[J. F. Solanilla]{Jos\'e F. Solanilla Hern\'andez}
\address{J. F. Solanilla Hern\'andez:  Centro Regional Universitario de Cocl\'e: \lq\lq Dr. 
Bernardo Lombardo\rq\rq,  Universidad de  Panam\'a. Apartado Postal 0229. Penonom\'e, 
Provincia de Cocl\'e. Panam\'a.}
\email{jose.solanilla@up.ac.pa}

\begin{abstract} In this paper we introduce new techniques in order to deepen into the structure of a Leavitt path algebra with the aim of giving a description of the center. Extreme cycles appear for the first time; they concentrate the purely infinite part of a Leavitt path algebra and, jointly with the line points and vertices in cycles without exits, are the key ingredients in order to determine the center of a Leavitt path algebra. Our work will rely on our previous approach to the center of a prime Leavitt path algebra \cite{CMMSS1}. We will go further into the structure itself of the Leavitt path algebra. For example, the ideal  $I(P_{ec} \cup P_{c} \cup P_l)$ generated by vertices in extreme cycles ($P_{ec}$), by vertices in cycles without exits ($P_c$) and by line points ($P_l$) will be a dense ideal in some cases, for instance in the finite one or, more generally, if every vertex connects to $P_l \cup P_c\cup P_{ec}$. Hence its structure will contain much of the information about the Leavitt path algebra. In the row-finite case, we will need to add a new hereditary set: the set of vertices whose tree has infinite bifurcations ($P_{b^\infty}$).
\end{abstract}
\maketitle

\section{Introduction and preliminary results}

When trying to determine the structure of a Leavitt path algebra $L_K(E)$, one can realize that two essential pieces appear. These are the sets of line points, $P_l(E)$, which are the vertices whose tree does not contain neither bifurcations nor cycles, and the set of vertices in cycles without exits, $P_c(E)$. 

The ideal generated by $P_l(E)$, isomorphic to a direct sum of matrix rings over $K$, is precisely the socle of the Leavitt path algebra (this was studied in \cite{AMMS1, AMMS2, ARS}), so it contains the locally  artinian side of the Leavitt path algebra. On the other hand, $P_c(E)$ contains the information about the locally noetherian character of the Leavitt path algebra: the ideal generated by $P_c(E)$ is isomorphic to a direct sum of matrix rings over $K[x, x^{-1}]$; this was determined in \cite{AAPS, ABS}. 

There is however a third ingredient whose presence could be guessed but which was immaterial until now: the purely infinite heart of the Leavitt path algebra. 

In this paper we introduce the notion of extreme cycle: a cycle with exits such that every path starting at the cycle connects to the cycle, and show that the ideal generated by the set $P_{ec}(E)$ of vertices in these cycles is a direct sum of purely infinite simple rings (see Section \ref{extreme}).

In Leavitt path algebras, density of an ideal generated by a hereditary subset of vertices, say $H$, can be translated graphically: every vertex of the graph connects to a vertex in $H$, as shown in Section 1 (Proposition \ref{density}).
Section 1 is also devoted to study ideals generated by the union and the intersection of hereditary subsets. 

We will see that $P_l(E)$, $P_c(E)$ and $P_{ec}(E)$ are the three primary colors of the center of a Leavitt path algebra and our intuition says that their importance goes further. This comes out also in Corollary \ref{structurePrime}, where we see that when a Leavitt path algebra $L_K(E)$ coming from a graph with a finite number of vertices is prime, every vertex connects to one and only one of the sets $P_l(E)$, $P_c(E)$ or $P_{ec}(E)$. In each case $I(P_l(E))$, $I(P_c(E))$ or $I(P_{ec}(E))$ is a dense ideal of $L_K(E)$, hence they contain the essential information about the Leavitt path algebra. The ideal  $I(P_{ec} \cup P_{c} \cup P_l)$ will be dense in some cases, for instance in the finite one or, more generally, if every vertex connects to $P_l \cup P_c\cup P_{ec}$. For row-finite graphs every vertex will connect to $P_l \cup P_c\cup P_{ec}\cup P_{b^\infty}$, hence the ideal it generates is dense.

The set $P=P_l(E)\cup P_c(E) \cup P_{ec}(E)$ determines if there exists nontrivial center in the Leavitt path algebra $L_K(E)$, but not only, because the center is related to finite subgraphs of $E$, and more concretely,  the cardinal of the equivalence classes determined by the relation given in Definitions \ref{classextendido} will provide the cardinal of the nonzero components of the center.

Concretely we show that for every row-finite graph $E$ and every field $K$,

$$Z(L_K(E))\cong K^{\vert{X}^l_f\vert} \oplus K^{\vert {X}^{ec}_f\vert}\oplus K[x, x^{-1}]^{\vert{X}^c_f\vert},$$

\noindent where ${X}^l_f, {X}^{ec}_f$ and ${X}^c_f$ are certain subsets of classes of elements in $P_l\cup P_c \cup P_{ec}$.

We give a basis for the center of the Leavitt path algebra of a row-finite graph and compute the extended centroid of the Leavitt path algebra of a finite graph.

\medskip

Now, we start with the basic definitions.

\medskip

A \emph{directed graph} is a 4-tuple $E=(E^0, E^1, r_E, s_E)$ consisting of two disjoint sets $E^0$, $E^1$ and two maps
$r_E, s_E: E^1 \to E^0$. The elements of $E^0$ are called the \emph{vertices} of $E$ and the elements of $E^1$ the edges of $E$ while for
$e\in E^1$, $r_E(e)$ and $s_E(e)$ are called the \emph{range} and the \emph{source} of $e$, respectively. If there is no confusion with respect to the graph we are considering, we simply write $r(e)$ and $s(e)$.

Given a (directed) graph $E$ and a field $K$, the {\it path $K$-algebra} of $E$,
denoted by $KE$ is defined as the free associative $K$-algebra generated by the
set of paths of $E$ with relations:
\begin{enumerate}
\item[(V)] $vw= \delta_{v,w}v$ for all $v,w\in E^0$.
\item [(E1)] $s(e)e=er(e)=e$ for all $e\in E^1$.
\end{enumerate}

 If
$s^{-1}(v)$ is a finite set for every $v\in E^0$, then the graph is called \emph{row-finite}. If
$E^0$ is finite and $E$ is row-finite, { then} $E^1$ must necessarily be finite as well; in this case we
say simply that $E$ is \emph{finite}.

A vertex which emits no edges is called a \emph{sink}. 
A vertex $v$ is called an \emph{infinite emitter} if $s^{-1}(v)$ is an infinite set, and a \emph{regular vertex} otherwise. 
The  set of infinite emitters {will} be denoted by $ E_{inf}^0$ while ${\rm Reg}(E)$ will denote the set of regular vertices.

The  {\it extended graph of} $E$ is defined as the new graph $\widehat{E}=(E^0,E^1\cup (E^1)^*, r_{\widehat{E}}, s_{\widehat{E}}),$ where
$(E^1)^*=\{e_i^* \ | \ e_i\in  E^1\}$ and the functions $r_{\widehat{E}}$ and $s_{\widehat{E}}$ are defined as 
$${r_{\widehat{E}}}_{|_{E^1}}=r,\ {s_{\widehat{E}}}_{|_{E^1}}=s,\
r_{\widehat{E}}(e_i^*)=s(e_i), \hbox{ and }  s_{\widehat{E}}(e_i^*)=r(e_i).$$

\noindent
The elements of $E^1$ will be called \emph{real edges}, while for $e\in E^1$ we will call $e^\ast$ a
\emph{ghost edge}.    
\medskip

The {\it Leavitt path algebra of} $E$ {\it with coefficients in} $K$, denoted $L_K(E)$, is the quotient of the path algebra $K\widehat{E}$ by the ideal of $K\widehat{E}$ generated by the relations:

\begin{enumerate}    
\item[(CK1)] $e^*e'=\delta _{e,e'}r(e) \ \mbox{ for all } e,e'\in E^1$.
\item[(CK2)] $v=\sum _{\{ e\in E^1\mid s(e)=v \}}ee^* \ \ \mbox{ for every}\ \ v\in  {\rm Reg}(E).$
\end{enumerate}

Observe that in $K\widehat{E}$ the relations (V) and (E1) remain valid and that the following is also satisfied:

\begin{enumerate}
\item [(E2)] $r(e)e^*=e^*s(e)=e^*$ for all $e\in E^1$.
\end{enumerate}
\medskip

 Note that if $E$ is a finite graph, then
$L_{K}(E)$ is unital with $\sum _{v\in E^0} v=1_{L_{K}(E)}$; otherwise, $L_{K}(E)$
is a ring with a set of local units consisting of sums of distinct vertices (for a ring $R$
the assertion \emph{$R$ has local units} means that each
finite subset of $R$ is contained in a \emph{corner} of $R$, that
is, a subring of the form $e Re$ where $e$ is an idempotent of $R$).
 Note that since every Leavitt path algebra $L_{K}(E)$ has
local units, it is the directed union of its corners.

A \emph{path} $\mu$ in a graph $E$ is a finite sequence of edges $\mu=e_1\dots e_n$
such that $r(e_i)=s(e_{i+1})$ for $i=1,\dots,n-1$. In this case, $s(\mu):=s(e_1)$ and $r(\mu):=r(e_n)$ are the
\emph{source} and \emph{range} of $\mu$, respectively, and $n$ is the \emph{length} of $\mu$. We also say that
$\mu$ is \emph{a path from $s(e_1)$ to $r(e_n)$} and denote by $\mu^0$ the set of its vertices, i.e.,
$\mu^0:=\{s(e_1),r(e_1),\dots,r(e_n)\}$. By $\mu^1$ we denote the set of edges appearing in $\mu$, i.e., $\mu^1:=\{e_1,\dots, e_n\}$.

We view the elements of $E^{0}$ as paths of length $0$. The set of all paths of a graph $E$ is denoted by ${\rm Path}(E)$.

The Leavitt path algebra $L_{K}(E)$ is a
$\mathbb{Z}$-graded $K$-algebra, spanned as a $K$-vector space by
$\{\alpha\beta^{\ast } \ \vert \ \alpha, \beta \in {\rm Path}(E)\}$. In particular, for each $n\in\mathbb{Z}$,
the degree $n$ component $L_{K}(E)_{n}$ is spanned by the set 
$\{\alpha \beta^{\ast }\ \vert \  \alpha, \beta \in {\rm Path}(E)\ \hbox{and}\  \mathrm{length}(\alpha)-\mathrm{length}(\beta)=n\}$. Denote by $h(L_K(E))$ the
set of all homogeneous elements in $L_K(E)$, that is, 
$$h(L_K(E)):= \cup_{n\in\mathbb{Z}}L_{K}(E)_{n}.$$

\smallskip

If $\mu$ is a path in $E$, and if $v=s(\mu)=r(\mu)$, then $\mu$ is called a \emph{closed path based
at $v$}. If $s(\mu)=r(\mu)$
and $s(e_i)\neq s(e_j)$ for every $i\neq j$, then $\mu$ is called a \emph{cycle}. A graph which
contains no cycles is called
\emph{acyclic}.  For $\mu = e_1 \dots e_n\in {\rm Path}(E)$ we write $\mu^*$ for the element $e_n^* \dots e_1^*$ of $L_{K}(E)$.

An edge $e$ is an {\it exit} for a path $\mu = e_1 \dots e_n$ if there exists $i\in \{1, \dots, n\}$ such that
$s(e)=s(e_i)$ and $e\neq e_i$. We say that $E$ satisfies \emph{Condition} (L) if every cycle in
$E$ has an exit. 
We denote by $P_c(E)$ ($P_c$ if there is no confusion about the graph)  the set of vertices of a graph $E$ lying in cycles without exits, and decompose it as:
$P_c(E)=P_c(E)^+ \sqcup P_c(E)^-$, where $P_c(E)^+$ are those elements of $c^0$, for $c$ a cycle without exits, such that
the number of paths ending at a vertex of $c^0$ and not containing all the edges of $c$ is infinite,  
 $P_c(E)^- =  P_c(E)\setminus  P_c(E)^+$ and $\sqcup$ denotes the disjoint union. If it is clear from the context the graph we are referring to, we will write simply $P_c^+$ or $P_c^-$.

Finally, given paths $\alpha, \beta$, we say $\alpha \leq \beta$ if $\beta =\alpha\alpha'$ for some path $\alpha'$.

Let $X$ be a subset of $E^0$. A \emph{path in $X$} is a path $\alpha$ in $E$ with
$\alpha^0\subseteq X$. We say that a path $\alpha$ in $X$ has \emph{an exit in $X$} if there
exists $e\in E^1$ { which is}  an exit for $\alpha$ { and} such that $r(e)\in X$.

\medskip

We define a relation $\ge$ on $E^0$ by setting $v\ge w$ if there exists a path in $E$ from $v$ to
$w$. A subset $H$ of $E^0$ is called \emph{hereditary} if $v\ge w$ and $v\in H$ imply $w\in H$. A
hereditary set is \emph{saturated} if every regular vertex which feeds into $H$ and only into $H$ is again
in $H$, that is, if $s^{-1}(v)\neq \emptyset$ is finite and $r(s^{-1}(v))\subseteq H$ imply $v\in H$. Denote
by $\mathcal{H}_E$ the set of hereditary saturated subsets of $E^0$.
\medskip

Hereditary and saturated subsets of vertices play an important role in the theory of Leavitt path algebras; in fact, they are closely related to graded ideals of Leavitt path algebras, as was highlighted for the first time in \cite{AMP}; and also to ideals since every ideal $I$ in a Leavitt path algebra $L_K(E)$ contains a graded part, the ideal generated by $I\cap E^0$ (see \cite[Theorem 2.8.6]{AAS}).
For a graph $E$ the ideal generated in $L_K(E)$ by a subset $X$ of vertices in $E^0$, denote it by $I(X)$, is a graded ideal, as it is generated by elements of degree zero. Moreover, if $E$ is a row-finite graph every graded ideal $J$ of $L_K(E)$ is $I(H)$ for $H$ a hereditary and saturated subset of $E^0$; concretely, $H=J\cap E^0$ (see \cite[Lemma 2.1 and Remark 2.2]{APS1}). Although not every hereditary subset has to be saturated, 
hereditary subsets (much more easy to get) give important information about the Leavitt path algebra.
\medskip

Whenever $X$ is a  set of vertices of a graph $E$, the \emph{saturated closure} of $X$ is defined as $\cup_{i\in \mathbb{N}}\Lambda_i(X)$, where $\Lambda_0(X)=X$ and by recurrence 
$\Lambda_i(X)= \Lambda_{i-1}(X)\cup \{v\in  {\rm Reg}(E)\ \vert \ r(s^{-1}(v))\in \Lambda_{i-1}(X)\}$. In particular, 
for a hereditary subset of vertices, say $H$, this saturated closure is hereditary and saturated and is denoted by $\overline H$. 

For $X$ a subset of vertices in a graph $E$, the \emph{hereditary closure} of $X$ is defined as the minimum hereditary subset of $E^0$ containing $X$. It always exists because is just the intersection of all hereditary subsets of $E^0$ which contains $X$. The \emph{hereditary and saturated closure} of a set of vertices is defined as the saturated closure of the hereditary closure.

Results that will be very useful are the following.

\begin{lemma}\label{Interseccion}
Let $E$ be an arbitrary graph and let $H_1, H_2$ be non empty hereditary subsets of vertices of $E$. Then:
\begin{enumerate}[\rm (i)]
\item $\Lambda_m(H_i)$ is hereditary for every $m\in \mathbb{N}$.
\item $\overline {H_1} \cap \overline {H_2}=\overline{H_1\cap H_2}$. 
\end{enumerate}
\end{lemma}
\begin{proof}
(i). For $m=0$ the result  is trivial. Suppose the result true for $m-1$ and let us show it for $m$. Take $u\in \Lambda_m(H_i)$ and let $e\in E^1$ be such that $s(e)=u$. Then $r(e)\in r(s^{-1}(u))\subseteq \Lambda_{m-1}(H_i)\subseteq \Lambda_{m}(H_i)$. This implies the result.

(ii). It is immediate to see $\overline{H_1\cap H_2}\subseteq \overline {H_1} \cap \overline {H_2}$. For the converse we will prove: $\Lambda_m(H_1)\cap \Lambda_m(H_2)= \Lambda_m(H_1\cap H_2)$. Note that the first observation implies 
$\Lambda_m(H_1)\cap \Lambda_m(H_2)\supseteq \Lambda_m(H_1\cap H_2)$. For the converse containment, use induction. If $m=0$ then the result is trivially true. Suppose our assertion is true for $m-1$ and show it for $m$. If $u\in \Lambda_m(H_1)\cap \Lambda_m(H_2)$ then $u\in \Lambda_{m-1}(H_1)$ or
$r(s^{-1}(u))\subseteq \Lambda_{m-1}(H_1)$. In the first case, and since $\Lambda_{m-1}(H_1)$ is hereditary (by (i)) we have also $r(s^{-1}(u))\subseteq \Lambda_{m-1}(H_1)$. Analogously we prove $r(s^{-1}(u))\subseteq \Lambda_{m-1}(H_2)$. This means $r(s^{-1}(u))\subseteq \Lambda_{m-1}(H_1)\cap \Lambda_{m-1}(H_2)= \Lambda_{m-1}(H_1\cap H_2)$ (by the induction hypothesis) and so $u\in \Lambda_{m}(H_1\cap H_2)$.
\end{proof}

\begin{lemma}\label{caminillos}
Let $E$ be a graph and $H$ a hereditary subset of $E^0$. Then, for every $v\in \overline H$ there exists a finite number of paths $\alpha_1,\dots, \alpha_n$ satisfying $r(\alpha_i)\in H$ and
$v=\sum_{i=1}^n\alpha_i\alpha_i^*$.
\end{lemma}
\begin{proof} For $v$ in $H$ we get immediately the result.
Take $v$ in $\Lambda_1(H)$ not being a sink. Then $v=\sum_{f\in s^{-1}(v)} ff^*$ and we have the claim. Suppose the result true for every $u\in \Lambda_{i-1}(H)$ and take $v\in \Lambda_i(H)$. Then, for every $f\in s^{-1}(v)$, since $r(f)\in r(s^{-1})(v)\subseteq \Lambda_{i-1}(H)$, by the induction hypothesis,  there exists a finite number of paths $\beta_1^f, \dots, \beta_m^f$ in ${\rm Path}(E)$ such that $r(f)=\sum_i\beta_i^f(\beta_i^f)^*$ and $r(\beta_i^f)\in H$. Hence $v=\sum_{f\in s^{-1}(v)}ff^*= \sum_f f(\sum_i\beta_i^f(\beta_i^f)^*)f^*
= \sum_{f, i}f\beta_i^f(\beta_i^f)^*f^*$ and we have finished.
\end{proof}

The following result was stated in \cite[Proposition 3.1]{AMMS1} (althought in that paper the statement is slightly different); it  has proved to be very useful in many different contexts, for example in order to get the Uniqueness Theorems (see \cite[Theorem 3.5]{AMMS1}). In this paper we also find another context where it can be used. 
\medskip

Except otherwise stated, $E$ will denote an arbitrary graph and $K$ an arbitrary field. As usual,  we will use the notation $K^\times $ for $K\setminus \{0\}$.

\begin{theorem}\label{red}
For every nonzero element $a$ in a Leavitt path algebra $L_K(E)$, there exist $\alpha, \beta\in {\rm Path}(E)$ such that:
\begin{enumerate}[\rm (i)]
\item $0\neq \alpha^\ast a \beta=kv$ for some $k\in K^\times$ and $v\in E^0$, or
\item $0\neq \alpha^\ast a \beta=p(c,c^\ast)$, where $c$ is a cycle without exits in $E$ and $p(x, x^{-1})\in K[x, x^{-1}]$.
\end{enumerate}
\end{theorem}

Of special interest will be the following result. 

\begin{corollary}\label{redhom}
Let $a$ be a nonzero homogeneous element in a Leavitt path algebra $L_K(E)$. Then, there exist $\alpha, \beta\in {\rm Path}(E)$, $k\in K^\times$ and $v\in E^0$, such that $0\neq \alpha^\ast a \beta=kv$.
\end{corollary}

\begin{proof}
Let $\alpha, \beta\in {\rm Path}(E)$ be such that $0\neq \alpha^\ast a \beta$ is as in cases (i) or (ii) in Theorem \ref{red}. In the first case, we have finished. In the second one, use the grading to obtain that in fact $p(c, c^\ast)$ in (ii) has to be a monomial, that is, $\alpha^\ast a \beta = kc^m$ for some $k\in K^\times$ and a certain integer $m$. If $m=0$, there is nothing more to do. If  $m>0$, then$ ({c^\ast})^m\alpha^\ast a \beta =kr(c)$ and we are done. If $m<0$, then $\alpha^\ast a \beta c^{-m}=kr(c)$ and the proof is complete.
\end{proof}


Another useful result which derives from Theorem \ref{red} is:

\begin{corollary}\label{vertH} Let $H$ be a non-empty hereditary subset of a graph $E$. Then, for every nonzero homogeneous $a\in I(H)$ there exist $\alpha, \beta \in {\rm Path}(E)$ such that $\alpha^\ast a \beta =kv$ for some $k\in K^\times$ and $v\in H$.
\end{corollary}
\begin{proof} Given the nonzero element $a\in I(H)$ apply 
Corollary \ref{redhom} and choose $\lambda, \mu\in {\rm Path}(E)$ such that $\lambda^\ast a \mu =kw$ for some $k\in K^\times$ and $w\in E^0$. 
Observe that $w\in I(H)$.
Use \cite[Lemma 3.1]{AAPS}  to write $w=\sum_{i=1}^m k_i\lambda_i\mu_i^\ast$
with $k_i \in K^\times$, $\lambda_i, \mu_i \in {\rm Path}(E)$, $r(\lambda_i)=r(\mu_i)\in H$, and suppose $\lambda_i\mu_i^* \neq \lambda_j\mu_j^*$ for every $i\neq j$.
Then for $v= r(\mu_1)$, $\alpha=\lambda\mu_1$ and $\beta=\mu\mu_1$ we have
$\alpha^*a\beta= \mu_1^\ast\lambda^\ast a \mu\mu_1 = k \mu_1^\ast w\mu_1 = k\mu_1^\ast \mu_1= kr(\mu_1)=kv$, which is nonzero and satisfies $v\in H$. 
\end{proof}

\begin{proposition}\label{sumher}
Let $\{H_i\}_{i\in \Lambda}$ be a family of hereditary subsets of a graph $E$ such that $H_i \cap H_j =\emptyset$ for every $i\neq j$. Then: 
$$I\left(\overline{\underset{i\in \Lambda}\cup H_i}\right) =I\left(\underset{i\in \Lambda}\cup H_i\right) = \underset{i\in \Lambda}\oplus I(H_i)
= \underset{i\in \Lambda}\oplus I\left(\overline{H_i}\right).$$
\end{proposition}
\begin{proof}
The union of any family of hereditary subsets is again hereditary, hence $H:=\underset{i\in \Lambda}\cup H_i$ is a hereditary subset of $E^0$. By \cite[Lemma 3.1]{AAPS} every element $a$ in $I(H)$ can be written as $a=\sum_{l=1}^nk_l\alpha_l\beta^\ast_l$, where $k_l\in K^\times$, $\alpha_l, \beta_l \in {\rm Path}(E)$ and
$r(\alpha_l)=r(\beta_l)\in H$. Separate the vertices appearing as ranges of the $\alpha_l$'s depending on the $H_i$'s they belong to, and apply again \cite[Lemma 3.1]{AAPS}. This gives
$a \in \underset{i\in \Lambda}\sum I(H_i)\subseteq I(H)$ since $H_i\subseteq H$ and so $\underset{i\in \Lambda}\sum I(H_i)= I(H)$.

Now we prove that the sum of the $I(H_i)$'s is direct. If this is not the case, since we are dealing with graded ideals, we may suppose that there exists a homogeneous element $0\neq a\in I(H_j)\cap \underset{j\neq i\in \Lambda}\sum I(H_i)$ for some
$j\in \Lambda$; by Corollary \ref{vertH} there exist $\alpha, \beta \in {\rm Path}(E)$ and $k\in K^\times$ such that $0\neq k^{-1}\alpha^\ast a \beta =w\in H_j$. Observe that $w$ also belongs to $I(\underset{j\neq i\in \Lambda}\cup H_i)$.

Write $w=\sum_{l=1}^nk_l\alpha_l\beta^\ast_l$, with $k_l\in K$, $\alpha_l, \beta_l \in {\rm Path}(E)$,
$r(\alpha_l)=r(\beta_l)\in \underset{j\neq i\in \Lambda}\cup H_i$, and assume that every summand is non-zero. 
Then $0\neq r(\beta_1)= \beta_1^\ast \beta_1= \beta_1^\ast w\beta_1\in \underset{j\neq i\in \Lambda}\cup H_i$. On the other hand, $s(\alpha_1)=w\in H_j$ implies (since $H_j$ is a hereditary set)  $r(\alpha_1)\in H_j$; therefore, $r(\alpha_1)=r(\beta_1)\in H_j \cap \left(\underset{j\neq i\in \Lambda}\cup H_i\right)$, a contradiction.

To conclude the proof we point out that the first and last identities follow from \cite[Lemma 2.1]{APS1}.
\end{proof}

A similar relation can be established for the ideal generated by the intersection of a family of hereditary subsets.
\medskip

\begin{lemma}\label{interseccionhereditarios} 
Let $\{H_i\}_{i\in \Lambda}$ be a family of hereditary  subsets of an arbitrary graph $E$. Then:
\begin{enumerate}[\rm (i)]
\item $I(\cap_{i\in \Lambda}\overline{H_i})= \cap_{i\in \Lambda}I(H_i)$.
\item If $\Lambda$ is finite, then $I(\cap_{i\in \Lambda}H_i)= \cap_{i\in \Lambda}I(H_i)$.
\end{enumerate}
\end{lemma}
\begin{proof} 
(i). \ By the isomorphism given in \cite[Theorem 2.4.13]{AAS} among hereditary and saturated subsets of vertices and a special type of graded ideals,  $I( \cap_{i\in \Lambda}\overline{H_i}) =  \cap_{i\in \Lambda}I(\overline{H_i}) = \cap_{i\in \Lambda}I(H_i)$.

(ii). \ When $\Lambda$ is finite, then $\cap_{i\in \Lambda} \overline{H_i} = \overline{\cap_{i\in \Lambda}H_i}$ (use Lemma \ref{Interseccion}), and consequently 
$$I(\cap_{i\in \Lambda}H_i)= I(\overline{\cap_{i\in \Lambda}H_i}) = \cap_{i\in \Lambda}I(\overline{H_i}) = \cap_{i\in \Lambda}I(H_i).$$

\end{proof}

Before stating the result that will allow us to work with connected graphs,
 we need to recall that the ideal generated by a hereditary and saturated subset in a Leavitt path algebra is isomorphic to a Leavitt path algebra.


\medskip


Let $E$ be a graph. For every non empty hereditary subset $H$ of $E^0$, define
$$F_E(H)=\{ \alpha =e_1 \dots e_n\mid e_i\in
E^1, s(e_1)\in E^0\setminus H, r(e_i)\in E^0\setminus H \mbox{ for } i<n, r(e_n)\in H\}.$$ Denote by $\overline{F}_E(H)$ another
copy of $F_E(H)$. For $\alpha\in F_E(H)$, we write $\overline{\alpha}$ to denote a copy of $\alpha$ in $\overline{F}_E(H)$. Then, we define the graph
$${}_HE=({}_HE^0, {}_HE^1, s', r')$$

\noindent   
 as follows:
\begin{enumerate}

\item ${}_HE^0=({}_HE)^0=H\cup F_E(H)$. \item ${}_HE^1=({}_HE)^1=\{ e\in
E^1\mid s(e)\in H\}\cup \overline{F}_E(H)$.
\item For every $e\in E^1 \mbox{ with } s(e)\in H$, $s'(e)=s(e)$ and $r'(e)=r(e)$. \item For every $\overline{\alpha}\in
\overline{F}_E(H)$,  $s'(\overline{\alpha})=\alpha$ and $r'(\overline{\alpha})=r(\alpha)$.
\end{enumerate}
\medskip

The following result was first proved in \cite[Lemma 5.2]{APS1} and then in \cite[Lemma 1.2]{AP} for row-finite graphs, although the result is valid in general (see \cite[Theorem 2.4.22]{AAS}).

\medskip

\begin{lemma}\label{idhersetLPA} Let $E$ be an arbitrary graph and $K$ any field. For a hereditary subset  $H\subseteq E^0$, the ideal
$I(H)$ is isomorphic  to the Leavitt path algebra $L_K({_H}E)$. Concretely, there is an isomorphism (which is not graded)
$\varphi: L_K({_H}E) \to I(H)$ acting as follows: 
\begin{enumerate}
\item[ ] $\varphi (v) = v$ for every $v\in H$,
\item[ ] $\varphi(\alpha)=\alpha\alpha^*$ for every $\alpha \in F_E(H)$,
\item[ ] $\varphi(e)=e$ and $\varphi(e^*)=e^*$ for every $e\in E^1$ such that $s(e)\in H$.
\item[ ] $\varphi (\overline \alpha)=\alpha$ and $\varphi (\overline \alpha^*)=\alpha^*$ for every $\overline\alpha \in
\overline{F}_E(H)$.
\end{enumerate}
\end{lemma}

%
%

When  we build the Leavitt path algebra of a graph $E$, we consider paths not only in $E$, but in the extended graph $\widehat{E}$; this means that when we think of a connected graph we have in mind ghost paths too. For this reason we say that a graph $E$ is {\it connected} if $\widehat{E}$ is a connected graph in the usual sense, that is, if given any two vertices $u, v\in E^0$ there exist $h_1, \dots, h_m\in E^1\cup ({E^1})^\ast$ such that $\eta:=h_1\dots h_m$ is a path in $K\widehat{E}$ (in particular it is non-zero) such that $s(\eta)=u$ and $r(\eta)=v$.

The {\it connected components} of a graph $E$ are the graphs $\{E_i\}_{i\in \Lambda}$ such that $E$ is the disjoint union 
$E=\sqcup_{i\in \Lambda}  E_i$, where every $E_i$ is connected.

\begin{corollary}\label{ConnectedComp} Let $E$ be a graph and suppose $E=\sqcup_{i\in \Lambda} E_i$, where each $E_i$ is a connected component of $E$. Then
$L_K(E) \cong \oplus_{i\in \Lambda} L_K(E_i)$.
\end{corollary}
\begin{proof}
By Proposition \ref{sumher}, $L_K(E)= \oplus_{i\in \Lambda}I(E^0_i)$ and by 
Lemma \ref{idhersetLPA}, $I(E_i^0)$ is isomorphic to the Leavitt path algebra $L_K({_{E^0_i}}E)$. Since the graph ${_{E^0_i}}E$ is just $E_i$, the result follows.
\end{proof}

By means of this corollary, and for our purposes, from now on we will restrict our attention to Leavitt path algebras of connected graphs.
\medskip

Another application of Proposition \ref{sumher} is stated below. Concretely, we will see that essentiality of  graded ideals generated by hereditary and saturated subsets in a Leavitt path algebra (which is equivalent to density as one-sided ideals) can be expressed in terms of properties of the underlying graph. 
  
  \begin{proposition}\label{density}
  Let $H$ be a hereditary  subset of a graph $E$. Then $I(H)$ is a dense (left/right) ideal if and only if every vertex of $E^0$ connects to a vertex in $H$.
 \end{proposition}
 \begin{proof}
 We first remark that since every Leavitt path algebra is left nonsingular (see \cite[Proposition 4.1]{SilesAQLPA}), the notions of dense left/right ideal and that of essential are equivalent in this context (by \cite[(8.7) Proposition]{lam}), and by  \cite[(14.1) Proposition]{lam}, $I(H)$ is essential as a left/right ideal if and only if it is essential as an ideal. Moreover, as $I(H)$ is a graded ideal, by \cite[2.3.5 Proposition]{Nast}  essentiality  and graded essentiality of $I(H)$ are equivalent. Hence, we will show that $I(H)$ is a graded essential ideal if and only if every vertex of $E^0$ connects to a vertex in $H$.
 
 Suppose first that $I(H)$ is a graded essential ideal of $L_K(E)$. Let $v\in E^0$. If $H \cap T(v) = \emptyset$, then Lemma \ref{interseccionhereditarios} would imply $I(H)\cap I(T(v)) = 0$, but this cannot happen as $I(H)$ is a graded essential ideal. Hence $H\cap T(v) \neq \emptyset$. This implies that $v$ connects to a vertex in $H$.
 
 Now we prove the converse, i.e.,  that $I(H)$ is an essential graded  ideal. Let $J$ be a nonzero graded ideal and pick a nonzero homogeneous element $x= uxv\in J$, where $u, v \in E^0$. Since the Leavitt path algebra $L_K(E)$ is an algebra of right quotients of $KE$, by  \cite[Proposition 2.2]{msm} (which is valid even for non necessarily row-finite graphs) there exists $\mu \in {\rm Path}(E)$ such that $0\neq x\mu\in KE$. Denote by $w$ the range of $\mu$. 
 By the hypothesis $w$ connects to a vertex in $H$, hence there exists $\lambda \in {\rm Path}(E)$ such that $w=s(\lambda)$ and $r(\lambda)\in H$. If $x\mu\lambda=0$ then $x\mu \in uL_K(E)w\cap KE$ would satisfy $\lambda\in {\rm Path}(E)\cap ran(x\mu)=\emptyset$,  by \cite[Lemma 1]{CMMSS1}, a contradiction, hence $0\neq x\mu\lambda\in I(H)\cap J$ which shows our claim.
  \end{proof}

%
\section{Extreme cycles}\label{extreme}
%

In this section we introduce the notion of extreme cycle. Roughly speaking it is a cycle such that every path starting at a vertex of the cycle comes back to it. The ideal generated by extreme cycles will be proved to be  a direct sum  of purely infinite simple Leavitt path algebras.

 \begin{definitions}\label{ece}{\rm
 Let $E$ be  a graph and $c$ a cycle in $E$. We say that $c$ is an {\it extreme cycle} if $c$ has exits and for every path $\lambda$ starting at a vertex in $c^0$ there exists $\mu\in \rm{Path}(E)$ such that $0\neq \lambda\mu$ and $r(\lambda\mu)\in c^0$. We will denote by $P_{ec}(E)$ the set of vertices which belong to extreme cycles.
 } 
 \end{definitions}

 \begin{definitions}\label{eceClass}{\rm 
  Let $X'_{ec}$ be the set of all extreme cycles in a graph $E$. We define in $X'_{ec}$ the following relation: given $c, d\in X'_{ec}$, we write $c \sim d$ whenever $c$ and $d$ are connected, that is, $T(c^0)  \cap d^0  \neq \emptyset$, equivalently, $T(d^0)  \cap c^0 \neq \emptyset$.
 It is not difficult to see that $\sim$ is an equivalence relation. Denote the set of all equivalence classes by $X_{ec} =X'_{ec}/\sim$. 
 When we want to emphasize the graph we are considering we will write $X'_{ec}(E)$ and $X_{ec}(E)$ for $X'_{ec}$ and $X_{ec}$, respectively.
 
 For any $c\in X'_{ec}$, let $\tilde c$ denote the class of $c$ and let use ${\tilde c}^0$ to represent the set of all vertices which are in the cycles belonging to $\tilde c$.}
\end{definitions} 
 The following will be of use.
 
 \begin{remark}\label{propECE} \rm{For a graph $E$ and using the notation described above:
 \begin{enumerate}[\rm (i)]
 \item For any $c\in X'_{ec}$, ${\tilde c}^0 =T(c^0)$, hence ${\tilde c}^0$ is a hereditary subset.
 \item  Given $c, d\in X'_{ec}$, $\tilde c \neq \tilde d$ if and only if ${\tilde c}^0 \cap {\tilde d}^0 = \emptyset$.
 \item For $c, d\in X'_{ec}$, $c\sim d$ if and only if $T(c) = T(d)$.
 \end{enumerate}}
 \end{remark}
 
  \begin{examples}
 Consider the following graphs 
 
$$ E \equiv \quad \quad
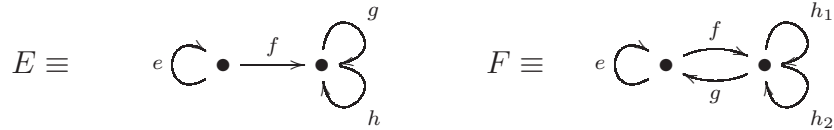
\xymatrix{\bullet \ar@(dl,ul)^e   \ar [r]^f  & {\bullet}  \ar@(u,r)^g    \ar@(r,d)^h } 
 \quad \quad \quad
 F \equiv \quad
  \xymatrix{\bullet \ar@(dl,ul)^{e} \ar@/^.5pc/[r]^{f} & \bullet \ar@/^.5pc/[l]^{g} \ar@(u,r)^{h_1} \ar@(r,d)^{h_2}}$$
\smallskip 

Then $X'_{ec}(E)=\{g, h\}$, $X_{ec}(E)=\{\tilde g\}$, $X'_{ec}(F) = \{e, fg, gf, h_1, h_2\}$ and $X_{ec}(F)=\left\{\tilde{e}\right\}$.
 \end{examples}
  \medskip
  
  Now we will analyse the structure of the ideal generated by $P_{ec}(E).$
  \medskip

  \begin{lemma}\label{eceIsPIS}
  Let $E$ be an arbitrary graph and $K$ any field. For every cycle $c$ such that $c\in X_{ec}'$, the ideal $I(\tilde c^0)$ is  isomorphic to a purely infinite simple Leavitt path algebra. Concretely, to $L_K({_H}E)$, where $H= {\tilde c}^0$.
\end{lemma}
\begin{proof}
Use Lemma \ref{idhersetLPA} to have $I(\tilde c^0)$ isomorphic to the Leavitt path algebra $L_K({_H}E)$ and let us show that this Leavitt path algebra is purely infinite and simple. For that, we will use the characterization  \cite[Theorem 4.3]{AA3}, which is valid for arbitrary graphs as can be proved by using the techniques described in \cite{G} or in \cite{ABS} in order to translate certain characterizations of  Leavitt path algebras from the row-finite case to the case of an arbitrary Leavitt path algebra.

Every vertex of ${_H}E$ connects to a cycle: take $v\in {_H}E^0$; if $v\in H$ then it connects to $c$, otherwise there exists a path $\mu\in {\rm Path}(E)$ such that $s(\mu)=v$ and $r(\mu)\in H$. Since $r(\mu)$ connects to the cycle $c$ in $E$, the vertex $v$ also connects to $c$ in ${_H}E$.

Every cycle in ${_H}E$ has an exit: this follows because any cycle in this graph comes from a cycle $d$ in $H$ and, by construction, $\tilde d=\tilde c$; this means that $d$ connects to $c$ and hence it has an exit which is an exit in ${_H}E$.

The only hereditary and saturated subsets of ${_H}E^0$ are $\emptyset$ and ${_H}E^0$: let $H'\in \mathcal{H}_{{_H}E}$ be non empty and consider $v \in H'$; if $v\in H$ then $H\subseteq H'$; since $H'$ is saturated, $H'={_H}E^0$; if $v\notin H$ then there exists $f\in {_H}E^1$ such that $v=s(f)$ and $r(f)\in H$; this implies $H\subseteq H'$; now, apply the construction of ${_H}E$ and that $H'$ is saturated to get $H'={_H}E^0$.
\end{proof}
  
  \begin{proposition}\label{DescriptionPece} 
  Let $E$ be any graph and $K$ any field. Then $I(P_{ec}(E))=\oplus_{\tilde c\in X_{ec}}I({\tilde c}^0)$ and every $I({\tilde c}^0)$ is isomorphic to a Leavitt path algebra which is purely infinite simple.
  \end{proposition}
\begin{proof}
The hereditary set $P_{ec}(E)$ can be decomposed as: $P_{ec}(E)=\sqcup_{\tilde c\in X_{ec}}{\tilde c}^0$. By the Remark \ref{propECE} and the Proposition \ref{sumher}, $I(P_{ec}(E))= I(\sqcup_{\tilde c\in X_{ec}}{\tilde c}^0)=\oplus_{\tilde c\in X_{ec}}I({\tilde c}^0)$. Finally, observe that every 
$I({\tilde c}^0)$ is isomorphic to a purely infinite simple Leavitt path algebra by Lemma \ref{eceIsPIS}.
\end{proof}

\begin{lemma} \label{intersection}
For any graph $E$ the hereditary sets $P_l(E)$, $P_c(E)$ and $P_{ec}(E)$ are pairwise disjoint. Moreover, 
the ideal generated by their union is $I(P_l(E))\oplus I(P_c(E))\oplus I(P_{ec}(E))$.
\end{lemma}
\begin{proof} By the definition of $P_l(E)$, $P_c(E)$ and $P_{ec}(E)$, they are pairwise disjoint. To get the result, apply Proposition \ref{sumher}.
\end{proof}

The following ideal, will be of use in the rest of the paper, so we name it here.
\medskip

\begin{definition}\label{lce} For a graph $E$ we define
$$I_{lce}: =I(P_l(E))\oplus I(P_c(E))\oplus I(P_{ec}(E)).$$ 
\end{definition}

\begin{theorem}\label{idealDenso}   
Let $E$ be an arbitrary graph and $K$ any field.
Then:
\begin{enumerate}[\rm (i)]
\item $I_{lce}\cong \left(\oplus_{i\in \Lambda_1} M_{m_i}(K)\right) \oplus \left(\oplus_{j\in \Lambda_2} M_{n_j}(K[x, x^{-1}]) \right)\oplus \left(\oplus_{l\in \Lambda_3}I(\tilde{c}_{n_l}^0)\right)$, where $\Lambda_1$ is the index set of the sinks of $E$, $\Lambda_2$ is the index set of the cycles without exits in $E$ and $\Lambda_3$ indexes $X_{ec}(E)$.
\item  If $\vert E^0\vert < \infty$ then $I_{lce}$ is a dense ideal of $L_K(E)$.
\end{enumerate} 
\end{theorem}
\begin{proof}
For our purposes we may suppose that the graph $E$ is connected because if $E=\sqcup_{i\in \Lambda}  E_i$ is the decomposition of $E$ into its connected components and we show the claims for every connected component, then the result will be true for $E$ by virtue of Corollary \ref{ConnectedComp}.

(i). We know that $I(P_l(E))$ is the socle of the Leavitt path algebra $L_K(E)$ (see \cite[Theorem 5.2]{AMMS2} and \cite[Theorem 1.10]{ARSs}), moreover $I(P_l(E))\cong \oplus_{i\in \Lambda_1} M_{m_i}(K)$, where $\Lambda_1$ ranges over the sinks of $E$ and for any $i\in \Lambda_1$, if $s_i$ is a sink (finite or not), then $m_i$ is the cardinal of the set of paths ending at $s_i$. 

The structure of $I(P_c(E))$ is also known: by \cite[Proposition 3.5]{AAPS} (which also works for arbitrary graphs) $I(P_c(E))\cong \oplus_{j\in \Lambda_2} M_{n_j}(K[x, x^{-1}])$, where the cardinal of $\Lambda_2$ is the cardinal of the set of cycles without exits and $n_j$ is the cardinal of the set of paths ending at a given cycle without exits $c_j$ and not containing the edges appearing in the cycle. 

Finally, the structure of the third summand in $I_{lce}$ follows by Proposition \ref{DescriptionPece}.

(ii). Take a vertex $v\in E^0$. Since $E^0$ is finite then $v$ connects to a line point, to a cycle without exists or to an extreme cycle. This means that  every vertex of $E$ connects to the hereditary set $H:=P_l(E) \cup P_c (E) \cup P_{ec}(E)$. By Proposition \ref{density} this means that $I(H)$ is a dense ideal of $L_K(E)$ and by Lemma \ref{intersection} it coincides with $I_{lce}$.
\end{proof}

\begin{remark}\label{obserP}{\rm
Following a similar reasoning as in the proof of Theorem \ref{idealDenso} it can be shown that any ideal of $L_K(E)$ generated by vertices in $P$ is isomorphic to 
$$\left(\oplus_{i\in \Lambda_1} M_{m_i}(K)\right) \oplus \left(\oplus_{j\in \Lambda_2} M_{n_j}(K[x, x^{-1}]) \right)\oplus \left(\oplus_{l\in \Lambda_3}I(\tilde{c}_{n_l}^0)\right)$$
for some $\Lambda_1$, $\Lambda_2$, $\Lambda_3$.}
\end{remark}

We specialize this result in the case of a prime Leavitt path algebra.

\begin{corollary}\label{structurePrime}
Let $E$ be a graph with a finite number of vertices such that $L_K(E)$ is a prime algebra. Then:
\begin{enumerate}[\rm (i)] 
\item There is a unique sink $v$ in $E$ and every vertex of $E$ connects to $v$. In this case $I(P_l(E))=I(T(v))\cong M_m(K)$, where $m$ is the cardinal of the set of paths ending at $v$, or
\item there is a unique cycle without exits $c$ and every vertex of $E$ connects to it. In this case $I(P_c(E))=I(c^0)\cong M_n(K[x, x^{-1}])$, where $n$ is the cardinal of the number of paths ending at $c$ and not containing all the edges of $c$, or
\item $ X_{ec}(E) =\{ \tilde c\}$, where $c$ is an extreme cycle and every vertex of $E$ connects to $c$. In this case $I(P_{ec}(E))= I(c^0)$ is a purely infinite simple unital ring.
\end{enumerate}
Moreover, the three cases are mutually exclusive. 
\end{corollary}

%
\section{The center of the Leavitt path algebra of a row-finite graph}\label{centrofinito} 
%

In this section our aim is to completely describe the center of the Leavitt path algebra of a row-finite graph.
The key pieces for such a description are: the set of line points, the vertices in cycles without exits and the vertices in extreme cycles, jointly with an equivalence relation defined on $E^0$ whose set of classes is indexed in a subset of $P:=P_l(E)\cup P_c(E) \cup P_{ec}(E)$.

The extended centroid of the ideal generated by $P$ will coincide with the center of the Martindale symmetric ring of quotients of the Leavitt path algebra, notion that plays an important role.
\medskip

Recall that for an associative algebra $A$, the {\it center} of $A$, denoted $Z(A)$, is defined by:

$$Z(A):=\{x\in A \ \vert \ [x, a]=0 \ \hbox{for every}\ a \in A \},$$

\noindent
where  $[a, b]:=ab-ba$ and juxtaposition stands for the product in the algebra $A$.
\medskip

For a semiprime algebra $A$, the {\it extended centroid} of $A$, denoted by $\C(A)$ is defined as:

$$\C(A) = Z(Q_s(A)) = Z(Q_{max}^l(A)) =  Z(Q_{max}^r(A)),$$

\medskip
\noindent
where $Q_s(A)$, $Q_{max}^l(A)$ and $Q_{max}^r(A)$ are the Martindale symmetric ring of quotients of $A$, the maximal left ring of quotients of $A$ and the maximal right ring of quotients of $A$, respectively (see \cite[(14.18) Definition]{lam}).

\begin{lemma}\label{centroidPIS}
Let $A$ be a unital   simple algebra. Then $Z(A)=C(A)$.
\end{lemma} 
\begin{proof}
We see first $Z(A)\subseteq \C(A)$. Suppose this containment is not true. Then there exists $x\in Z(A)$ and $q\in Q_s(A)$ such that $xq-qx \neq 0$.
Use that $Q_s(A)$ is a right ring of quotients of $A$ to find $a\in A$ such that $0\neq (xq-qx)a$ and $qa\in A$. Then $0\neq x(qa)-q(xa)= (qa)x-q(ax) =0$, a contradiction.

To prove $\C(A)\subseteq Z(A)$, consider $q\in \C(A)\setminus \{0\}$. By \cite[(14.22) Corollary]{lam}, $\C(A)$ is a field, hence there exist $q^{-1}\in \C(A)$. Use again that $Q_s(A)$ is a right ring of quotients of $A$ to find $x\in A$ such that $0\neq q^{-1}x\in A$. Since $A$ is simple and unital 
$Aq^{-1}xA=A$, in particular there exists a finite number of elements $a_1, \dots, a_m, b_1, \dots, b_m \in A$ such that $1=\sum_{i=1}^m a_iq^{-1}xb_i$.
Multiply this identity by $q$ and use that $q$ is in the center of $Q_s(A)$ to get $q=\sum_{i=1}^m a_ixb_i\in A$ as desired.
\end{proof}
\medskip

The following remarks and results  will be useful to study the center of a Leavitt path algebra. Their proofs are straightforward.

\begin{remark}\label{centrosumasdirectas}
{\rm If $A$ is an algebra and $\{I_i\}$ is a set of ideals of $A$ whose sum is direct, then $Z(\oplus_i I_i)= \oplus_i(Z(I_i))$. 
}
\end{remark}

\begin{lemma}\label{centrograd}
Let $G$ be an abelian group. The center of a ${G}$-graded algebra $A$  is  ${G}$-graded, that is, if $x\in Z(A)$ and $x=\sum_{g\in {G}}x_g$ is the decomposition of $x$ into its homogenous components, then $x_g\in Z(A)$ for every $g\in {G}$.
\end{lemma}
\begin{proof}
Indeed, for every $y=\sum_{h\in {G}}y_h$ in $A$, 
$0=[x,\ y_h]= [\sum_{g\in {G}} x_g,\ y_h]= \sum_{g\in {G}} [x_g,\ y_h]$
 and using the grading on $A$ we get 
$[x_g,\ y_h]=0$ for any $g\in {G}$. Hence $0=\sum_{h\in {G}}[x_g,\ y_h] =[x_g,\ \sum_{g\in {G}}y_h]=[x_g,\ y]$ which means 
$x_g\in Z(A)$.
\end{proof}

\begin{notation}{\rm  The homogeneous component of degree $g$ in the center of a $G$-graded algebra $A$ will be denoted by $Z_g(A)$.}
\end{notation}

\begin{lemma}\label{idealcentro}
Let $I$ be an ideal of an algebra $A$. If for every $y\in Z(I)$ there exist $n\in \mathbb{N}$ and $\{a_i, b_i\}_{i=1}^n\subseteq Z(I)$ such that $y=\sum_{i=1}^na_ib_i$, then $Z(I)=I\cap Z(A)$.
\end{lemma}
\begin{proof} It is clear that $I\cap Z(A)\subseteq Z(I)$. To show $Z(I)\subseteq Z(A)$, take 
 $y\in Z(I)$ and $x\in A$. Write $y=\sum_{i=1}^na_ib_i$, for $a_i, b_i\in Z(I)$. Then $yx=\sum_{i=1}^na_ib_ix=\sum_{i=1}^na_i(b_ix)=\sum_{i=1}^n(b_ix)a_i= \sum_{i=1}^nb_i(xa_i)= \sum_{i=1}^n(xa_i)b_i=xy$. 
\end{proof}

\begin{corollary}\label{centroideal}
Let $I$ be an ideal of a Leavitt path algebra $L_K(E)$ such that for every $y\in Z(I)$ there exist $a, b \in Z(I)$ such that $y=ab$.
 Then $Z(I)=I\cap Z(L_K(E))$. This happens, in particular, for every ideal of $L_K(E)$ generated by vertices in $P$.
\end{corollary}
\begin{proof} 

The first statement follows immediately from Lemma \ref{idealcentro} For $I$ generated by vertices in $P$, by Remark \ref{obserP},  $I$ is a direct sum of ideals of $L_K(E)$ which are isomorphic to $M_n(K)$, $M_n(K[x, x^{-1}])$ or $J$, for $J$ purely infinite and simple. In this last case, by \cite[Theorem 3.6]{CMMSS1}, $Z(J)$ is $0$ or isomorphic to $K$;  in the other cases, the centers are zero or isomorphic to $K$, or to $K[x, x^{-1}]$. In all of these situations our hypothesis on the ideal is satisfied.
\end{proof}

\begin{theorem}\label{centroide}
Let $E$ be graph such that $\vert E^0\vert < \infty$ and consider $I: =I_{lce}$. Then the extended centroid of $L_K(E)$ coincides with the extended centroid of $I$, $\mathcal{C}(I)$; moreover,  
$$\mathcal{C}(L_K(E)) = \mathcal{C}(I)\cong  \left(\oplus_{i=1}^m K\right) \oplus  \left(\oplus_{j=1}^n K[x, x^{-1}]\right)\oplus \left(\oplus_{l=1}^{n'} K\right),$$
 where $m$ is the number of sinks, $n$ is the number of cycles without exits and $n'$ is the number of equivalence classes of extreme cycles.
If $P_l(E)$, $P_c(E)$ or $P_{ec}(E)$ are empty, then the ideals they generate are zero and the corresponding summands in $\C(I)$ do not appear.
\end{theorem}
\begin{proof}
As in the proof of Theorem \ref{idealDenso} we may suppose that our graph is connected.
Apply Theorem \ref{idealDenso} (ii) and \cite[(14.14) Theorem]{lam} to obtain $Q_s(L_K(E))=Q_s(I)$. Then by  \cite[Lemma 1. 3 (i)]{msm} $ \mathcal C(L_K(E))={\mathcal C}(I)= {\mathcal C}\left(I(P_l(E))\right)\oplus  {\mathcal C}\left(I(P_c(E))\right) \oplus   {\mathcal C}\left(I(P_{ec}(E))\right)$.

By Theorem \ref{idealDenso} (i), Lemma \ref{centroidPIS} and Proposition \ref{DescriptionPece}, $\mathcal C (I(P_{ec}(E))) = \oplus_{l=1}^{n'} \C(I(\tilde c_l^0))$, where $\tilde c_l\in X_{ec}$ and $n'=\vert X_{ec}\vert$. Use \cite[Theorem 4.2]{AC} to obtain $\C(I(\tilde c_l^0))\cong K$ for every $l$. 

To finish, use the three pieces of information in the paragraphs before to obtain $C(L_K(E))=\mathcal C(I)=\left(\oplus_{i=1}^m K\right) \oplus  \left(\oplus_{j=1}^n K[x, x^{-1}]\right)\oplus \left(\oplus_{l=1}^{n'} K\right)$ and we get the claim in the statement.
\end{proof}
\medskip

\begin{definitions}\label{classextendido}{\rm 
Let  $u,v \in E^0$. The element $\lambda_{u,v}$ will denote a path such that $s(\lambda_{u,v})=u$ and $r(\lambda_{u,v})=v.$

 In $E^0$ we define the following relation: given $u,v \in E^0$ we write $u\sim^1 v$ if and only if $u=v$ or: 
\begin{enumerate}
\item [\rm (i)] $u\leq v$ or $v\leq u$ and there are no bifurcations at any vertex in $T(u)$ and $T(v)$.
\item [\rm (ii)]
there exist a cycle $c$, a vertex $w\in c^{0}$ and  $\lambda_{w,u}$,$\lambda_{w,v}$  $\in Path(E)$.
\end{enumerate}
This relation $\sim^1$ is reflexive and symmetric. Consider the transitive closure of  $\sim^1$; we shall denote it  by $\sim$. The notation $[v]$ will stand for the class of a vertex $v$.}
\end{definitions}

\begin{example} {\rm
The vertices $u$ and $v$ in the following graph are related:
$$
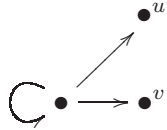
\xymatrix@-0.5pc{
& &\bullet^{u} & &\\
&\bullet\ar@(ul,dl)[] \ar[r]\ar[ur]&\bullet^{v} &&\\
&&&&}
$$}
\end{example}
\begin{remark} {\rm
The following is an example of a graph which illustrates why do we need to consider the transitive closure of the relation $\sim^1$ in Definition \ref{classextendido}. Note that $u\sim^1 v$, $v \sim^1 w$ and $u\not\sim^1 w$; however, $u\sim w$. 

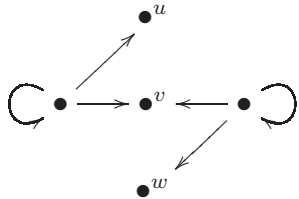
$$\xymatrix@-0.5pc{
& &\bullet^{u} & &\\
&\bullet\ar@(ul,dl)[] \ar[r]\ar[ur]&\bullet^{v} & \bullet \ar@(ur,dr)[]\ar[l]\ar[dl]&\\
&&\bullet^{w}&&}
$$
}
\end{remark}

In what follows we are going to describe the zero component of the center of a Leavitt path algebra $L_K(E)$ associated to a row-finite graph.

\begin{notation} {\rm Let $E$ be an arbitrary graph. Consider $P=P_l\cup P_c \cup P_{ec}$ and  define $X=P/\sim$.
Decompose $P=P_f\sqcup P_\infty$, where $P_f$ are those elements $v$ of $P$ such that 
\begin{enumerate}[\rm (i)]
\item $\vert \overline{[v]}\vert < \infty $, and
\item $\vert F_E( \overline{[v]})\vert < \infty $
\end{enumerate}
In the same vein we decompose $X=X_f\sqcup X_\infty$, where 
$$X_f=\{[u]\in X\ \vert \ \text{for all}\ v\in [u],\ v\in P_f  \}$$
and
$$X_\infty= \{[u]\in X\ \vert \ \text{for some}\ v\in [u],\ v\in P_\infty \}.$$

Finally, we decompose $X_f=X_f^l \cup X_f^c \cup X_f^{ec}$, where each of these subsets consists of equivalence classes induced by elements which are in $P_f\cap P_l$, in $P_f \cap P_c$ and in $P_f \cap P_{ec}$, respectively. Note that if $u$ and $v$ are vertices in $P_f$, then $u\sim v$ if and only if $u, v\in P_l$, $u, v \in P_c$ or $u, v\in P_{ec}$. 

When we want to emphasize the graph $E$ we are considering, we will write $P_l(E), X(E)$, etc.
}
\end{notation}

\begin{lemma}\label{disyuncion}
Let $E$ be an arbitrary graph and $u, v \in P$. Then:

\begin{enumerate}[\rm (i)]
\item  $[u]$ is a hereditary set.
\item If $u\not\sim v$ then  $\overline {[u]} \cap \overline{[v]} = \emptyset$.
\end{enumerate}
\end{lemma}
\begin{proof}
(i). Let $ w\in [u]$ and consider $w'\in r(s^{-1}(w))$. Since $w\sim u$ there exists a finite set 
$\{v_1\dots v_n\}$ of vertices such that $w=v_1\sim^1v_2\sim^1\dots \sim^1v_n=u$. Note that $ w'\sim^1v_2$ and so $w'\in [u]$.   

(ii). By the hypothesis $[u]\cap [v]=\emptyset$. Use (i) and Lemma \ref{Interseccion} to get the result.
\end{proof}

\begin{definition}\label{estandarf}{\rm
Let $E$ be a graph and $K$ be any field. An element $a\in L_K(E)$ which can be written as
$$a=\sum_{{[v]\in {X_f}}}k_{[v]}a_{[v]}, \quad \text{where}\quad  k_{[v]}\in K^\times \quad \text{and} \quad
a_{[v]}=\sum_{u\in \overline{[v]}}u +\sum_{\alpha\in F_E({\overline{[v]}})}\alpha\alpha^*,$$
will be said to be written in the \it{standard form}.}
\end{definition}

We will prove that every element in the zero component of the center of a Leavitt path algebra can be written in the standard form.

\begin{lemma}\label{multipillin} 
Let   $E$ be an arbitrary graph and $K$ be any field. Consider $[v]\in X$. For every $u\in \overline{[v]}$ and  $\alpha, \beta \in F_E({\overline{[v]}})$ we have: 
\begin{enumerate}[\rm (i)]
\item If $s(\alpha)=s(\beta)$, then $ \alpha^*\beta \neq 0$ if and only if $\alpha =\beta$.
\item  If $s(\alpha)\neq s(\beta)$ then $ \alpha^*\beta= 0$.
\item $u\alpha=0$.
\end{enumerate}
\end{lemma}
\begin{proof}
(i). If $\alpha^*\beta \ne 0$ then $\alpha=\beta \gamma$ or  $\beta=\alpha\delta $ for some $\gamma, \delta \in \text {Path}(E)$. By the definition of $ F_E({\overline{[v]}})$, necessarily $\alpha=\beta$.

(ii). This case follows immediately.

(iii). For $u$ and $\alpha$ as in the statement, $u\alpha\neq 0$ implies $u=s(\alpha)$, but this is not possible as $\alpha\in F_E(\overline{[v]})$.
\end{proof}

\medskip

\begin{notation}{\rm
For a graph $E$ we denote by ${P}_e$ the set of vertices in cycles with exits.
}
\end{notation}
\begin{lemma}\label{capullete}
Let  $E$ be an arbitrary graph. Then, for every $a \in Z_0$ and $v\in P\cup P_e$ there exists $k_v\in K$ such that if $u \in \overline{[v]}$ then $uau=k_vu$.
\end{lemma}
\begin{proof} Suppose first $u= v$.
If $v\in P_c \cup P_{e}$ the result follows by \cite[Corollary 7]{CMMSS1}.
If $v\in P_l$, by the proof of \cite[Proposition 1.8]{ARSs} $vL_K(E)v=Kv$, hence $av=va\in Kv$ .

Now, suppose $u\in [v]$. Since the relation $\sim$ is given as the transitive closure of $\sim^1$, we may assume first that there exists a vertex $w$ in a cycle, and paths $\lambda, \mu$ satisfying $\lambda = w\lambda u$ and $\mu= w\mu v$. By the paragraph before there exists an element $k\in K$ such that $waw= kw$. Then $uau= ua= \lambda^*\lambda a =
\lambda^*a\lambda= \lambda^*(kw)\lambda= k \lambda^*\lambda = k u$ and analogously
we show $va=kv$. Repeating this argument a finite number of steps we reach our claim.

It is easy to see that when $u$ is in the hereditary closure of $[v]$ we also have the result.

Finally, take $u\in \overline{[v]}$. We may assume that $u$ is not a sink (this case has been studied yet).  By Lemma \ref{disyuncion} (i) and Lemma \ref{caminillos} we may write $u=\sum_{i=1}^n\alpha_i\alpha_i^*$, where the $\alpha_i$'s are paths and $r(\alpha_i)$ is in the hereditary closure of $[v]$. Then
$au=a\sum_{i=1}^n\alpha_i\alpha_i^*=\sum_{i=1}^n\alpha_ia\alpha_i^*=
\sum_{i=1}^n\alpha_ikr(\alpha_i)\alpha_i^*= k \sum_{i=1}^n\alpha_i\alpha_i^* = ku$.
\end{proof}

\begin{proposition}\label{tresdiecinueve}
Let   $E$ be a row-finite graph and consider a nonzero homogeneous element $a$ of degree zero in $Z(L_K(E))$. Then: 
\begin{enumerate}[\rm (i)]
\item $av=0$ for every $v\in P_\infty$.
\item $a=\sum_{{[v]\in {X_f}}}k_{[v]}a_{[v]} \quad \text{where}\quad
a_{[v]}=\sum_{u\in \overline{[v]}}u +\sum_{\alpha\in F_E({\overline{[v]}})}\alpha\alpha^*$
\end{enumerate}

\end{proposition}
\begin{proof}
(i). 
Suppose first $\vert \overline{[v]}\vert=\infty$. If $av\neq 0$, by Lemma \ref{capullete} we have $au\neq 0$ for every $u\in \overline{[v]}$, which is an infinite set; this is not possible, so $av=0$ and we have finished. Now, assume $\vert \overline{[v]}\vert<\infty$. Since $v\in P_\infty$, necessarily 
$\vert F_E({\overline{[v]}})\vert= \infty$, then we have an infinite collection of paths $\{\alpha_n\}_{n\in \mathbb{N}}\subseteq F_E({\overline{[v]}})$ with $\alpha_m\ne \alpha_n$ for $m\ne n$. Next we prove that $as(\alpha_n)\ne 0$ for each $n$. First, note that $r(\alpha_n)\in \overline{[v]}$; since $av\neq 0$ then $ar(\alpha_n)\neq 0$ by Lemma  \ref{capullete}, therefore $0\neq ar(\alpha_n)=a\alpha_n^*\alpha_n=\alpha_n^*as(\alpha_n)\alpha_n$ which implies our claim.

Since there is only a finite number of vertices $u\in E^0$ such that $au\neq 0$, the set
 $\{s(\alpha_n)\}_{n\in \mathbb{N}}$ must be finite. Moreover, the set $\overline{[v]}$, which contains all the ranges of the paths $\alpha_n$ is finite. Since $\vert \{\alpha_n^0\}_{n\in \mathbb{N}}\vert< \infty$ (because every vertex in $\alpha_n^0$ does not annihilate $a$ by Lemma \ref{capullete}), $v\notin P_f$ and no vertex in $\alpha_n^0\setminus r(\alpha_n)$ is an infinite emitter, we may conclude that there is a path $\alpha_m=\gamma_1 \dots \gamma_r \dots \gamma_t$, with $\gamma_i\in {\rm Path}(E)$ such that $l(\gamma_r)\geq 1$ and $s(\gamma_r)=r(\gamma_r)$, hence $\alpha_m$ contains a  cycle based at $s(\gamma_r)$, implying $s(\gamma_r)\in {[v]}$. This is a contradiction with the fact $\alpha_m\in F_E(\overline{[v]})$ and (i) has been proved.
 
 (ii). Write $a=\sum au$, with $u\in E^0$ and $au\neq 0$. If $u\sim v$, with $v\in P\cup P_e$,
 then $au=k_vu$ by Lemma \ref{capullete} and $k_v\in K^\times$. If $u\not\sim v$ for any  $v\in P\cup P_e$ then, as it is not a sink, by (CK2), we may write $u=\sum_{s(e)=u} ee^*$.
Then $au = a\sum_{s(e)=u} ee^* = \sum_{s(e)=u} eae^*$. Take $e$ in this summand. If $r(e)\in P\cup P_e$ then
$ar(e)=k_{r(e)}r(e)$ where $k_{r(e)} \in K$ and we get a summand as in the statement. Otherwise we apply (CK2) to $r(e)$ and write
$r(e)=\sum_{s(f)=r(e)} ff^*$; then $aee^*= ae\sum_{s(f)=r(e)} ff^*e^*$. Every summand $aeff^*e^*$ with $r(f)\in P\cup P_e$ is $k_{r(ef)}eff^*e^*$ by Lemma \ref{capullete}, which is a summand as in the statement. For every nonzero summand not being in this case we apply again (CK2). This process stops because otherwise we would have an infinite path $e_1e_2\dots$ with $e_i\in E^0$ such that $s(e_i)\neq s(e_j)$ for every $i\neq j$ and $ae_i\neq 0$ for every $i$, which is not possible as the number of vertices not annihilating $a$ has to be finite. Note that the path $e_1e_2\dots e_n$ we arrive at is, by construction, an element in $F_E({\overline{[v]}})$. We remark that $v\in P_f$ by (i).

Take $v\in E^0$ such that $[v]\in X_f$ and $av\neq 0$; by Lemma \ref{capullete} we have $av=k_vv$ for some $k_v\in K$. Note that $k_v\neq 0$. For any $u\in \overline{[v]}$, Lemma \ref{capullete} shows that $au=k_vu\neq 0$. If $\beta\in F_E(\overline{[v]})$, then by Lemma \ref{multipillin} $\beta^*a=k_{r(\beta)} \beta^*\beta\beta^*$. Since $\beta^*a=a\beta^*=ar(\beta)\beta^*=k_v\beta^*$, we get $k_{r(\beta)}=k_v\neq 0$. This shows that $a$ can be written as a linear combination of elements of the form 

$$(\dag)\quad  \sum_{u\in \overline{[v]}}u +\sum_{\alpha\in F_E({\overline{[v]}})}\alpha\alpha^*,$$
where $[v]\in X_f$. 
\end{proof}

\begin{theorem}\label{gradocero} Let $E$ be a row-finite graph. 
For every class $[v]\in X_f$, denote by $a_{[v]}=\sum_{u\in \overline{[v]}}u +\sum_{\alpha \in F_E(\overline{[v]})}\alpha\alpha^*$. Then
$$\mathcal{B}_0=\left\{
a_{[v]} \ \vert \ [v]\in X_f\right\}$$
is a basis of the zero component of the center of $L_K(E)$.
\end{theorem}
\begin{proof}
By Proposition \ref{tresdiecinueve} every element in the zero component of the center is a linear combination of the elements of $\mathcal{B}_0$. 
Lemmas \ref{disyuncion} and  \ref{multipillin} imply that $\mathcal{B}_0$ is a set of linearly independent elements. 

In what follows we show that every $a_{[v]}$ in $\mathcal{B}_0$ is in the center of $L_K(E)$.
To start, consider a vertex $w$ in $E^0$. It is not difficult to see $wa_{[v]}= a_{[v]}w$ since $a_{[v]}$ is in
$\oplus_{i=1}^n u_iL_K(E)u_i$ for a certain finite family of vertices $\{u_i\}_{i=1}^n$.

Consider and edge $e$ in $E$ and denote by $A:=\overline{[v]}\cup F_E(\overline{[v]})$. We claim:
$$(\dag) \quad r(e)\in A \quad \text{if and only if}\quad s(e)\in A\quad  \text{or}\quad e\in A.$$

Assume $r(e) \in A$. If $s(e) \notin A$ then $s(e)\notin \overline{[v]}$ and since $r(e) \in\overline{[v]}$  we conclude that $e \in F_E(\overline{[v]})\subseteq A$. Reciprocally, if $s(e) \in A$ then $r(e)\in A$ because $\overline{[v]}$ is hereditary (Lemma \ref{disyuncion} (i)).
Finally, if $e\in A$ then $e \in F_E(\overline{[v]})$ which implies $r(e) \in A$.

Now we see that $a_{[v]}$ commutes with $e$.
Suppose first that $r(e)\in A$. Then $ea_{[v]}=e$ by Lemma \ref{multipillin}. On the other hand, if $s(e)\in A$, then $a_{[v]}e=e+0=e$ by Lemma \ref{multipillin}; if $s(e)\notin A$, then $a_{[v]}e=0+e=e$ by Lemma \ref{multipillin}. Now, suppose $r(e)\notin A$. Then, by $(\dag)$ we have $s(e), e \notin A$. For each $\alpha$ we have $e\alpha=0$ or $e\alpha\neq 0$ and, in this last case, $e\alpha\in A$ since $s(e)\notin A$, by hypothesis, and $r(e)=s(\alpha)\notin A$.

Applying the results explained above, 
$$ea_{[v]}= e\sum_{\alpha\in F_E(\overline{[v]})}\alpha\alpha^*=
e\sum_{\overset{
\alpha\in F_E(\overline{[v]})}{\alpha^*e=0}}
\alpha\alpha^*+
e \sum_{\alpha\in F_E(\overline{[v]})}e\alpha\alpha^*e^*.$$
We claim that the second summand must be zero. Indeed, suppose $e^2\alpha\neq 0$; then, $s(e)\sim w$ for a vertex $w\in [v]$; this implies $s(e)\in [v]$, a contradiction since we assume $s(e)\notin A$.

Therefore

\begin{equation}\label{identi}  ea_{[v]}=
\sum_{\overset{
\alpha\in F_E(\overline{[v]})}{\alpha^*e=0}}
e\alpha\alpha^*
\end{equation}

In what follows we compute $a_{[v]}e$.

$$\begin{aligned}
a_{[v]}e & =  \left(\sum_{\overset{
\alpha\in F_E(\overline{[v]})}{\alpha^*e=0}}
\alpha\alpha^*\right)e+\left(
\sum_{\alpha\in F_E(\overline{[v]})}e\alpha\alpha^*e^*\right)e
= 
\sum_{\alpha\in F_E(\overline{[v]})}e\alpha\alpha^*e^*e\\
& = \sum_{\alpha\in F_E(\overline{[v]})}e\alpha\alpha^*
=
\sum_{\overset
{\alpha\in F_E(\overline{[v]})}{\alpha^*e = 0}}
e\alpha\alpha^*
+
\sum_
{\overset{\alpha\in F_E(\overline{[v]})}{\alpha^*e\neq 0}}
e\alpha\alpha^*
.\end{aligned}
$$
We claim that the second summand is zero. The reason is again that for $\alpha$ in the second summand, $e\alpha$ must be zero because $\alpha$ starts by $e$ and $\alpha \in F_E(\overline{[v]})$. Hence
$$a_{[v]}e = \sum_{\overset
{\alpha\in F_E(\overline{[v]})}{\alpha^*e = 0}}
e\alpha\alpha^*= ea_{[v]}$$
by (\ref{identi}).
\end{proof}

\medskip

\medskip

\begin{notation}\label{notati} {\rm 
Let $E$ be a graph. For a cycle $c=e_1\dots e_n$ in $E$ and $u_i=s(e_i)$ we will write $c_{u_i}$ to denote the cycle $e_i \dots e_{i-1}$. Let $\mathcal S$ be the set of all those cycles without exits $c$ such that there is a finite number of paths ending at $c$  and not containing $c$.

Note that if $c$ is a cycle in $\mathcal S$ and $u, v\in c^0$ with $u\neq v$, then $c_u$ and $c_v$ will be different elements in $\mathcal S$. 
}
\end{notation}

\begin{proposition}\label{Zgen} Let $E$ be an arbitrary graph and $K$ any field. Consider
a homogeneous element $a$ in $Z(L_{K}(E))$ with ${\rm deg}(a) > 0$.
  Then $au=0$ for all $u \in P_l\cup {P}_e$. If $u\in P_c$ then $au=k_uc_{u}^{r}$, where $k_u\in K$,  $c_u$ is a cycle without exits and $r\in \mathbb N$. Moreover, if $u\in P_c^+$, then $k_u=0$.
\end{proposition}
\begin{proof} 
If $u\in P_l$, then $au=uau\in uL_K(E)u =Ku$, by \cite[Proposition 4.7 and Remark 4.8]{AMMS2}, hence there exists $\lambda\in K$ such that $au=\lambda u$. Then $\lambda$ must be zero as the degree of $\lambda u$ is zero and ${\rm deg}(a) > 0$.
\newline

Now, take $u\in {P}_e$ and let $d$ be a cycle such that $u=s(d)$. 

We will use partially a reasoning that appears in \cite[Theorem 6]{CMMSS1}; we include it here for the sake of completeness.

 A generator system for $uL_K(E)u$ is $A\cup B$, for 

$$A= \left\{
d^n(d^m)^*\ \vert \ n, m \geq 0\right\}$$
and
$$
B= \left\{
d^n\alpha \beta^*(d^m)^*\ \vert \ n, m \geq 0,\ \alpha, \beta \in {\rm{Path}}(E),\ s(\alpha)=u=s(\beta), d \not\leq \alpha, \beta,\ \alpha^1\cup\beta^1\not\subseteq d^1 \right\}.$$
For $n=0$ we understand $d^n=u$. Note that given $n,m \geq 0$ and $d^n\alpha \beta^*(d^m)^*\in B$, there exists a  suitable $r\in \mathbb{N}$ such that $(d^r)^*d^n\alpha \beta^*(d^m)^*d^r=0$. This gives us that if we define the map $S:L_K(E) \to L_K(E)$  by
$S(x) = d^\ast x d$,  for every $b\in B$ there is an $n\in \mathbb{N}$ satisfying $S^n(b)=0$. Note that $au$ is a fixed point for $S$. A consequence of this reasoning is that $au\in span(A)$. Write $au=\sum_nk_{n}d^n(d^m)^*$, for $m=n-deg(a)$, where $k_n\in K$. Then, for some $l\in \mathbb{N}$ we have $au=S^l(au)=\sum_nk_nd^{deg(a)}$. Since $au$ commutes with every element in $uL_K(E)u$, the same should happen to $d^{deg(a)}$, but this is not true as it does not commute with $d^*$, giving $au=0$.
\newline

Next we show the second part of the statement. For $u\in P_c$, $au=uau\in uL_K(E)u\cong K[x,x^{-1}]$ (by \cite[Proposition 2.3]{ABS}). Since ${\rm deg}(a) > 0$, $au=k_uc_u^{r}$ for some
$r\in \mathbb{N}\setminus$\{0\}, $c_u$ a cycle without exits and $k_u\in K$.
\newline

To finish, consider $u\in P_c^+$ and write $au=k_u c^r$, for $c$ a cycle without exits and $k_u\in K$. Then, two cases can happen. First, assume that there exists a cycle $d$ such that $v:=s(d)\geq u$ and take a path $\alpha$ satisfying $\alpha= v\alpha u$. Then, as we have proved before, $av=0$ and so $k_uc^r= au= a\alpha^*\alpha = \alpha^*a\alpha=
\alpha^*av\alpha=0$; this implies $k_u=0$ as required. In the second case  there exists infinitely many paths $\gamma_n$, for $n\in \mathbb{N}$ ending at $v$. Since $E$ is row-finite (and we are assuming that we are not in the first case, so there are no cycles involved), $\vert \{s(\gamma_n)\}\vert = \infty$. Then there exists $m\in \mathbb{N}$ such that $as(\gamma_m)=0$ and so $av=a\gamma_m^\ast\gamma_m= \gamma_m^\ast as(\gamma_m)\gamma_m=0$ as desired.
\end{proof}
\medskip

\begin{lemma}\label{elsalvador} Let $E$ be a row-finite  graph and $K$ any field. For any
homogeneous element $a$ in $Z(L_{K}(E))$ with ${\rm deg}(a) > 0$ we have
$a=\sum a\alpha\alpha^*$, where $\alpha \in F_E(c^0)$ for some cycle without exits $c$ and $r(\alpha)\in P_c^-$.
\end{lemma}
\begin{proof} Let $u$ be in $E^0$ such that $au\neq 0$. 
Define 
$$T_1(u)=\{v\in T(u) \ \vert \ \exists \ e\in E^1,\ \text{such that}\ s(e)=u, r(e) =v\}.$$
Let $T_n(u)$ be
$$T_n(u)=\{v\in T(u)\setminus \left(\cup_{i< n} T_i(u)\right) \ \vert \ \exists \ \alpha\in {\rm Path (E)}, \ l(\alpha) = n,\ \text{such that}\ s(\alpha)=u, r(\alpha) =v\}.$$
\par

Given $0\neq au$, write $au= a\sum_{s^{-1}(u)=e} ee^*= \sum_{s^{-1}(u)=e} eae^*$ (note that $u$ cannot be a sink and so we may use (CK2)).
By Proposition \ref{Zgen}, the possibly nonzero summands are those $eae^*$ such that $r(e)\in T_1(u)\setminus (P_l \cup P_e\cup P_c^+)$.
Let $$S_1:=\{ aee^*\ \text{such that}\ r(e)\in P_c^-\}.$$
Now, to every element $aee^*$ with $r(e)\in T_1(u)\setminus (P_l \cup P_e\cup P_c)$, apply Condition (CK2). Then $aee^*= a(e\sum_{s^{-1}(r(e))=f} ff^*)e^*= (e\sum_{s^{-1}(r(e))=f} faf^*)e^*$. Again by Proposition \ref{Zgen} the nonzero summands are those $efaf^*e^*$ such that $r(f)\in T_2(u)\setminus (P_l \cup P_e\cup P_c^+)$.
Let $$S_2:=\{ aeff^*e^*\ \text{such that}\ r(f)\in P_c^-\}.$$

This reasoning must stop in a finite number of steps, say $m$, because otherwise we would have infinitely many edges $e_1, e_2\dots$ such that $ae_1\dotsc e_n\neq 0$ for every $n$. Since for every vertex $w$ in a cycle $aw=0$, the ranges of these paths are all different. Therefore, there would be infinitely many vertices not annihilating $a$, a contradiction. Note that, by construction, $a=\sum_{x\in \cup_{i=1}^mS_i} x$ and so $a$ can be written as in the statement, where $\alpha\in F_E(c^0)$ for some cycle without exits $c$.
\end{proof}
\bigskip

\begin{lemma}\label{multiplicacion}
Let $E$ be an arbitrary graph and consider   $c_u, d_w \in \mathcal{S}$, $\alpha \in {F_E(c_u^0)}$,  $\beta \in {F_E(d_w^0)}$ and  $n, m \in \mathbb Z$.  
\begin{enumerate}[\rm (i)]
\item If $w=u$, then
$ \alpha^*\beta \neq 0$ if and only if $\alpha =\beta$.
\item  If $w\neq u$ then $  c_u^m\alpha^*\beta d_w^n= 0$.
\end{enumerate}
\end{lemma}

\begin{proof}
(i). If $\alpha^*\beta \ne 0$ then $\alpha=\beta \gamma$ or  $\beta=\alpha\delta$ for some $\gamma, \delta \in \text {Path}(E)$.
In the first case, since $r(\beta)\in c_u^0$, by the very definition of ${F_E(c_u^0)}$ and taking into account $\alpha \in {F_E(c_u^0)}$ we get $\alpha=\beta$. In the second case we get analogously the same conclusion.
The converse is obvious.
\par
(ii). Take $w \neq u$ and assume $\alpha^*\beta \ne 0$; this implies as before  $\alpha=\beta \gamma$ or  $\beta=\alpha\delta$ for some $\gamma, \delta \in \text {Path}(E)$. In the first case, $r(\beta)\in d_w^0$; since $d_w$ has no exits, then $\gamma^0\subseteq d_w^0$. On the other hand, $r(\alpha)=r(\gamma)\in c_u^0\cap d_w^0$, therefore $c_u^0= d_w^0$ and arguing as in (i) we conclude $\alpha=\beta$. This implies $  c_u^m\alpha^*\beta d_w^n=  c_u^md_w^n= 0$ since $w\neq u$.

The other possibility  yields the same conclusion.
\end{proof}

The following result will relate elements in the $n$ component of the center to elements in the 0 component.

\begin{lemma}\label{aa*} Let $E$ an arbitrary graph and $K$ any field.
Let $a$ be a  nonzero  element in $Z(L_{K}(E))$. Then $aa^*\neq 0.$
\end{lemma}
\begin{proof} By Theorem \ref{red}, there are $\alpha, \beta \in Path (E)$ such that $0\neq \alpha^{*} a \beta = kv$ for some $k\in K^\times$, $v\in E^{0}$, or  $0\neq \alpha^{*} a \beta = p(c,c^*)$, where $p(c,c^*)$ is a polynomial in a cycle without exits $c$.

In the first case $0\neq k^2 v = (kv)(kv^*)=\alpha^{*} a \beta \beta^* a^*\alpha = \alpha^{*} \beta \beta^* aa^*\alpha $, since $a$ is in the center of $L_K(E)$,  and so $aa^*$ must be nonzero.

In the second case, $0\neq p(c, c^*)p(c, c^*)^*$ since $p(c, c^*)\in r(c)L_K(E)r(c)$ which is isomorphic to the Laurent polinomial algebra $K[x, x^{-1}]$ (see \cite[Proposition 2.3]{ABS}), hence $0\neq  \alpha^{*} a \beta\beta^*a^*\alpha = \alpha^{*}  \beta\beta^*aa^*\alpha $, and so $aa^*$ must be nonzero.
\end{proof}

\begin{theorem}\label{HP}Let $E$ be a row-finite graph. Then, the set 
$$\mathcal{B}_n=\left\{\sum_{\substack {m \cdot  l(c)=n \\ \alpha \in F_E(c^0)\cup \{c^0\}\\ u\in c^0}}\alpha c_u^m \alpha^*\quad \vert \quad c \in {\mathcal S}\right\}$$
is a basis of $Z_n(L_K(E))$ with $n\in \mathbb{Z}\setminus\{0\}$.
\end{theorem}
\begin{proof} We will assume $n>0$. The case $n<0$ follows by using the involution in the Leavitt path algebra.

Let $a \in Z_n(L_K(E))$. For every $u\in E^0$ such that $au\neq 0$, Lemma \ref{elsalvador} and Proposition  \ref{Zgen} imply

$$au=a\sum\alpha\alpha^{\ast}=\sum\alpha a r(\alpha)\alpha^{\ast}= \sum\alpha k_\alpha  c_{\alpha}^{m_{\alpha}}\alpha^{\ast}=\sum k_\alpha \alpha c_{\alpha}^{m_{\alpha}}\alpha^{\ast}$$
  
\noindent
where  $\alpha \in F_E(c^0)\cup \{u\},\ r(\alpha) \in P^-_c$ and $ s(\alpha)=u$.
 
Hence 

 \begin{equation}\label{expr}
 {a=\sum k_\alpha \alpha c_v^{m}\alpha^{\ast},  \text{where the $c_v$'s\ } \text{are elements of $\mathcal S, $ $\alpha \in F_E(c^0)\cup \{v\},\ r(\alpha) \in P^-_c$}.}
 \end{equation}
 
Now we see that 
$k_{\alpha}=k_{\beta}$ for every $\alpha, \beta \in F_E(c^0)$ appearing in the expresion before.

We first note that $\alpha c^r\beta^*$ is a nonzero element in $L_K(E)$ for every $r\in \mathbb N$. Since $a\alpha\beta^*=\alpha\beta^* a$, using Lemma \ref{multiplicacion} we get 
$k_\alpha \alpha c_u^m\beta^* = k_\beta \alpha c_u^m\beta^*$ and so $k_\alpha = k_\beta$.

In what follows we prove that if $c_u$ appears in ($\ref{expr}$) then for every $v\in c_u^0$ and for every $\beta\in F_E([c_u^0])$ we have that $c_v$ and $\beta c_v\beta^*$ also appear.

Apply  (\ref{expr}) and Lemmas \ref{aa*} and \ref{multipillin}  to get
 
$$0\neq aa^*=\left(\sum k_\alpha \alpha c_v^{m}\alpha^{\ast}\right)\left( \sum k_\alpha \alpha c_v^{-m}\alpha^{\ast}  \right) = \sum k_\alpha^2\alpha v\alpha^*.$$

By Theorem \ref{gradocero}, $v$ must appears in this summand and also every $\beta\in F_E([c_v^0])$.
This shows that the set $\mathcal{B}_n$ generates the $n$-component of the center. \par

In what follows we see that the elements of $\mathcal{B}_n$ are linearly independent. In fact, we show that elements of the form $\alpha c_u^m\alpha^*$ are linearly independent. To this end, let $k_\alpha \in K$ and suppose $ \sum k_\alpha \alpha c_u^{m_u}\alpha^* =0$, where all the summands are different. Choose a nonzero 
$\beta c_v^{m_v}\beta^*$. Then $0 =\beta^*\sum k_\alpha \alpha c_u^{m_u}\alpha^* =$ (by Lemma \ref{multiplicacion}) $\beta^*k_\beta \beta c_v^{m_v}\beta^* =  k_\beta c_v^{{m_v}}\beta^*$; this implies $k_\beta=0$
and our claim has been proved. 
\medskip

Finally we see that $\mathcal{B}_n\subseteq Z(L_K(E))$.
Fix $c_u	\in 	\mathcal{S}$ and denote by 
$A_n=\{\alpha c_u^m\alpha^* \ \vert\  m.l(c_u)=n,\ \alpha \in F_E([c_u^0])\cup \{u\}\}$, i.e., $A_n$ contains all the summands of $a$ in (\ref{expr}). We see that for every $\beta\in {\rm Path}(E)$, $A_n\beta = \beta A_n$ and $A_n\beta^\ast= \beta^\ast A_n$ for all $n\in \mathbb{Z}$; this will prove our statement. 

So, take $\alpha c_u^m\alpha^* \in A_n$. Then $\alpha c_u^m\alpha^* \beta$ is non zero if and only if $\beta=\alpha$ (by Lemma \ref{multiplicacion}); in this case, $\alpha c_u^m\alpha^* \beta=\beta c_u^m\in  \beta A_n$. Suppose $\alpha c_u^m\alpha^* \beta=0$; if $\beta \alpha c_u^m\alpha^*=0$ then we have shown that $0\in \beta A_n$;  otherwise $ \beta \alpha c_u^m\alpha^*\neq 0$; this implies $r(\beta)=s(\alpha)$. Now we distinguish two cases: first, $\alpha=u$. Note that $c_u^m\beta=0$ implies $\beta\neq c_u\gamma$, for any $\gamma\in {\rm Path}(E)$ and so $\beta c_u^m\beta^*\in A_n$.
Then $0=\beta\beta c_u^m\beta^*\in \beta A_n$ (because $r(\beta)=u$). If $\alpha$ is not a vertex, then 
$r(\beta)\neq u$  hence $0= \beta c_u^m\in \beta A_n$. 
 This shows $A_n\beta\subseteq \beta A_n$.
  \medskip
 
 For the converse, note that 
$\beta\alpha c_u^m\alpha^*$ is nonzero if and only if $\beta\alpha$ is a nonzero element in $F_E([c_u^0])$, for $w=s(\beta)$; in this case $\beta\alpha c_u^m\alpha^*\beta^*\in A_n$ and so $\beta\alpha c_u^m\alpha^*= \beta\alpha c_u^m\alpha^*\beta^*\beta \in A_n\beta$. Now, suppose $\beta\alpha c_u^m\alpha^*=0$. Then, multiplying on the right hand side by $\alpha{(c_u^m)}^*$ and on the left hand side by $\beta^*$ we get $r(\beta)\alpha=0$. If $c_u^m\beta =0$ we have shown $0=c_u^m\beta \in A_n\beta$. If $c_u^m\beta\neq 0$, as $c_u$ has no exits, then $\beta=c_u^s\lambda$, where $\lambda$ is such that $c_u$ starts by $\lambda$. If $l(c_u)=1$, then $\beta =c_u^s$ for some $s\in \mathbb{N}$; we see that this implies $\alpha c_u^m\alpha^*\beta=0$. Suppose otherwise $\alpha c_u^m\alpha^*\beta\neq 0$. Then $\alpha^*\beta\neq0$ and so $\alpha \geq \beta$ or $\beta \geq \alpha$. The first case is not possible (note that neither  $c_u^1\subseteq \alpha^1$ nor $\alpha=u$), hence $\beta \geq \alpha$, implying $s(\alpha)=u$. This happens only when $\alpha=u$, but this is not possible as $0=\beta\alpha c_u^m\alpha^*=c_u^{s+m}$ is a contradiction. Now, suppose $l(c) >1$. Let $e$ be the last edge in $c_u$. Then $ec_u^me^*\beta$, which is an element in $A_n\beta$, is zero (because the first edge of $\beta$ is the first one of $c_u$, which is not $e$). This proves $\beta A_n\subseteq A_n\beta$ and therefore $A_n\beta= \beta A_n$. Applying the involution we have $\beta^*A_{-n}=A_{-n}\beta^*$ for all $n\in \mathbb Z$, as required.
\end{proof}

\begin{definitions}\label{bifurcation}
Let $E$ be an arbitrary graph. Denote by $P_{b^\infty}$ the set of all vertices in $E$ such that $T(v)$ has infinite bifurcations. Define 
$$ H_f:= \bigsqcup_{[v]\in X_f}\overline{[v]}\quad \text{and} \quad H_\infty:= \left(\bigsqcup_{[v]\in X_\infty}\overline{[v]}\right) \cup P_{b^\infty}.$$
\end{definitions}

\begin{proposition}\label{sumadirecta} Let $E$ be a row-finite graph and $K$ be any field. Then 
$$L_K(E)= I(H_f)\oplus I(H_\infty).$$
\end{proposition} 
\begin{proof} We will show $v\in  I(H_f \cup H_\infty)$ for every $v\in E^0$. If $v\in H_f\cup H_\infty$ we have finished. Suppose that this is not the case. In particular, this means $s^{-1}(v)\neq \emptyset$. Write $v=\sum_{e\in s^{-1}(v)}ee^*$. If for every $e$ in this sum $r(e)\in H_f\cup H_\infty$, then we have finished. If for some $e$, $r(e)$ is a sink, we have finished; if $r(e)$ is a vertex in a cycle, taking into account that every vertex of $E^0$ connects to $H_f\cup H_\infty$ then $r(e)\sim u$ for some $u\in H_f\cup H_\infty$. Note that $u\notin P_{b^\infty}$ as otherwise $v\in P_{b^\infty}$, and we are assuming $v\notin H_f\cup H_\infty$, therefore $r(e)\in \cup_{w\in X}[w]$. For those $e\in s^{-1}(v)$ in the remaining cases we apply again Condition (CK2) to $u_1:=r(e)$ and write $u_1=\sum_{f\in s^{-1}(u_1)}ff^*$. Now we proceed in the same way concerning $r(f)$. This process must stop because otherwise there would be infinitely many bifurcations and so $v\in P_{b^\infty}$, a contradiction.

Apply Proposition \ref{sumher} to the disjoint hereditary subsets $H_f\cup H_\infty$ to get the result.
\end{proof}
\bigskip

\begin{theorem}\label{EstructuraCentroFinito}{\rm {\bf(Structure Theorem for the center of a row-finite graph).}}
Let $E$ be a row-finite graph. Then:
$$Z(L_K(E))\cong K^{\vert{X}\setminus {X}^c_f\vert} \oplus K[x, x^{-1}]^{\vert{X}^c_f\vert}$$
More concretely,

$$Z(L_K(E))\cong K^{\vert{X}^l_f\vert} \oplus K^{\vert {X}^{ec}_f\vert}\oplus K[x, x^{-1}]^{\vert{X}^c_f\vert}.$$
\end{theorem}
\begin{proof}
By Proposition \ref{sumadirecta},  $L_K(E)=I(H_f) \oplus I(H_\infty)$. Since  $H_f:= \sqcup_{[v]\in X_f}\overline{[v]}$ and $ H_\infty:= \left(\sqcup_{[v]\in X_\infty}\overline{[v]}\right) \cup P_{b^\infty}$, Proposition  \ref{sumher} implies 
$$L_K(E)=\left( \bigoplus_{[v]\in X_f}I(\overline{[v]})\right) \bigoplus I(H_\infty),$$

\noindent
hence $$Z(L_K(E))=\left( \bigoplus_{[v]\in X_f}Z(I(\overline{[v]})\right) \bigoplus Z(I(H_\infty)).$$

By Theorems  \ref{gradocero} and \ref{HP} we know $Z(L_K(E))\subseteq  \bigoplus_{[v]\in X_f}Z(I(\overline{[v]})$); therefore

 $$Z(L_K(E))= \bigoplus_{[v]\in X_f}Z(I(\overline{[v]}).$$

 We claim that $Z(I(\overline{[v]}))$ is isomorphic to $K$ if $v\in P_l\cup P_{ec}$ or isomorphic to $K[x, x^{-1}]$ in case $v\in P_c$.

Indeed, take $x\in Z(I(\overline{[v]}))$. By the theorems of the basis of the center we may write $x=\sum_{w}x_w$, where $[w]\in X_f$ and 
$$x_w=k_{w, 0}\ a_{[w]}+ \sum_{n}k_{w, n}\left(\sum_{{\substack {m \cdot  l(c_w)=n \\ \alpha \in F_E(c_w^0)\cup \{c_w^0\}\\ u\in c_w^0}}}\alpha c_u^m \alpha^*\right),$$

\noindent
 where $a_{[w]}$ is as in Theorem \ref{gradocero}, $k_{w, 0},\ k_{w, n} \in K$ and $k_{w, n}$ is zero for almost every $n\in \mathbb{Z}$ and $w\in P_c$. Note that  $x_w\in I(\overline{[w]})$ and that $\{I(\overline{[w]})\}_{[w]\in X_f}$ is an independent set (in the sense that its sum is direct). This means that, necessarily, $x=x_v$. 

If $v\in P_l\cup P_{ec}$, then $x=k_v^0a_{[v]}$ and we have an isomorphism of $K$-algebras spaces between $Z(I(\overline{[v]}))$ and $K$. If $v\in P_c$ then the following map:
$$a_{[v]}\mapsto 1 \quad \text{and} \quad
\sum_{\substack {m \cdot  l(c_w)=n \\ \alpha \in F_E(c_w^0)\cup \{c_w^0\}\\ u\in c_w^0}}\alpha c_u^m \alpha^* \mapsto x^m$$

\noindent
describes an isomorphism of $K$-algebras
between $Z(I(\overline{[v]}))$ and $K[x, x^{-1}]$. This finishes the proof.
\end{proof}  
\bigskip

\section*{Acknowledgments}
The authors would like to thank Gene Abrams and Pere Ara for
useful discussions.
\medskip

All the authors have been partially supported by the Spanish MEC and Fondos FEDER through project MTM2010-15223,  by the Junta de Andaluc\'{\i}a and Fondos FEDER, jointly, through projects FQM-336 and FQM-3737. The first and last authors have been partially supported also by the  programa de becas para estudios doctorales y postdoctorales SENACYT-IFARHU, contrato no. 270-2008-407,  Gobierno de Panam\'a and by the University of Panam\'a. This work was done
during research stays of the first and last author in the University of M\'alaga. Both
authors would like to thank the host center for its hospitality and support.


\end{document}